\documentclass[11pt]{amsart}
\usepackage{amssymb, amsmath, amsthm}
\usepackage{amsrefs}
\usepackage[latin1]{inputenc}
\usepackage{graphicx}
\usepackage[final]{hyperref}
\usepackage[a4paper, centering]{geometry}
\usepackage{color}
\usepackage{graphicx} 

\newtheorem{theorem}{Theorem}
\newtheorem{proposition}[theorem]{Proposition}
\newtheorem{lemma}[theorem]{Lemma}
\newtheorem{corollary}[theorem]{Corollary}
\theoremstyle{definition}
\newtheorem{definition}[theorem]{Definition}
\theoremstyle{remark}
\newtheorem{remark}[theorem]{Remark}
\newtheorem{example}[theorem]{Example}
\def\R{\mathbb{R}}

\def\N{\mathbb{N}}

\def\inradius{R_\Omega}

\def \e{\varepsilon} 

\def \Dinf{\Delta _\infty} 
\def \L{\Lambda}
\def \l {\lambda} 
\def \lp {\lambda_{\Omega,p}}
\def \Lp {\Lambda_{\Omega,p}}

\def\Vspace{\mathcal{V}}

\newcommand{\Oma}[1]{\left\{d > #1\right\}}
\newcommand{\Omc}[1]{\left\{d \geq #1\right\}}

\DeclareMathOperator{\dist}{dist}

\DeclareMathOperator{\interior}{int}

\begin{document}

\title[Bernoulli free boundary problem]%
{Bernoulli free boundary problem\\ for the infinity Laplacian}%
\author[G.~Crasta, I.~Fragal\`a]{Graziano Crasta,  Ilaria Fragal\`a}
\address[Graziano Crasta]{Dipartimento di Matematica ``G.\ Castelnuovo'', Univ.\ di Roma I\\
P.le A.\ Moro 2 -- 00185 Roma (Italy)}
\email{crasta@mat.uniroma1.it}

\address[Ilaria Fragal\`a]{
Dipartimento di Matematica, Politecnico\\
Piazza Leonardo da Vinci, 32 --20133 Milano (Italy)
}
\email{ilaria.fragala@polimi.it}


\date{May 8, 2019}

\begin{abstract}
We study the interior Bernoulli free boundary problem for the infinity Laplacian.  Our results cover existence, uniqueness, and characterization of solutions (above a threshold representing  the ``infinity Bernoulli constant''), their regularity, and their relationship with the solutions to the interior Bernoulli problem for the $p$-Laplacian. 
\end{abstract}

\maketitle

\section{Introduction}
This paper concerns  the following interior Bernoulli-type problem: 
\begin{equation}\label{f:P}
\begin{cases}
\Dinf u = 0 & \text{ in } \Omega^+(u) := \{x\in\Omega:\ u(x) > 0\},
\\
u = 1 & \text{ on } \partial \Omega,
\\ 
|\nabla u| = \lambda    
& \text{ on } F(u) := \partial\Omega^+(u) \cap \Omega\, , 
\end{cases} 
\tag*{$(P)_\lambda$}  
\end{equation}

where $\Omega$ is an open bounded connected domain  in  $\R ^n $ ($n\geq 2$),  
and $\Delta _\infty$ is the infinity Laplacian, defined by
$$\Delta _\infty u := \nabla ^ 2 u \nabla u \cdot \nabla u\qquad \forall u \in C ^ 2 (\Omega)\,.$$ 

Before presenting our results, we wish to put them into context by saying few words on related literature.

\subsection{Bernoulli problem for the $p$-Laplacian.}\label{s:ip}  The analogue of problem \ref{f:P} for the $p$-Laplacian,
namely
\[
\begin{cases}
\Delta_p u = 0
& \text{ in } \Omega^+(u)\,,
\\
u = 1 & \text{ on } \partial \Omega,
\\ 
|\nabla u| = \lambda    
& \text{ on } F(u)\,,
\end{cases}
\]
corresponds to the classical Bernoulli problem when $p=2$, and by now it has been widely studied also in the nonlinear case of an arbitrary $p >1$.  
It is motivated by several physical and industrial applications, 
including
fluid dynamics, optimal insulation, and electro-chemical machining (see \cite{FR97} for a more precise description).  
The main questions are the existence and uniqueness of solutions, the
geometric properties of the free boundary $F (u)$, and especially its regularity (for an overview  on these topics, we address to \cite{CaffSalsa,FS15}).

When $\Omega$ is convex and regular, it was proved by Henrot and Shahgholian  that there exists a positive constant 
$\lp$, called  the \textit{Bernoulli constant for the $p$-Laplacian}, such that  
the interior $p$-Bernoulli problem admits a non-constant solution 
if and only if $\lambda \geq \lp$; this solution is in general not unique, it 
has convex level sets, and its free boundary $F(u)$ is of class $C^{2,\alpha}$ (see \cite{HeSh2, CT02}).   

When $\Omega$ is an arbitrary domain, not necessarily convex, one way of finding solutions is to use the approach 
introduced by Alt and Caffarelli
in the linear case $p=2$
in the seminal work \cite{AlCa}. It 
amounts to minimizing the integral functional
\begin{equation}\label{f:Jpi}
J ^{\lambda } _p  (u) := \frac{1}{p} \int_\Omega \left(\frac{|\nabla u|}{\lambda} \right)^p  + \frac{p-1}{p} \,  \big | \{ u >0 \}   \big | 
\end{equation}
over the space  $W^{1,p}_1(\Omega) := 1 + W^{1,p}_0(\Omega)$.

This minimization problem admits a non-constant solution if and only if $\l \geq \Lp$, where 
$\Lp$ is a positive constant satisfying $\Lp \geq \lp$ \cite{DaKa2010}. 
 A non-constant minimizer of $J ^{\lambda } _p$ over  $W ^ {1,p}_1 (\Omega)$
 solves the $p$-Bernoulli problem provided the free boundary condition $|\nabla u| = \lambda$ is understood in a suitable weak sense
(cf.~\eqref{f:weakFB}). 
 The free boundary $F(u)$ turns out to be a locally analytic hyper-surface, except for a $\mathcal H ^{n-1}$-negligible singular set 
 (in the vast literature about the free boundary regularity, we limit ourselves to quote  as main contributions \cite{AlCa, CJK04, DSJ09} for the case $p = 2$ 
and \cite{DanPet, DanPet2, P08} for general $p$).

\smallskip 
\subsection{Free boundary problems for the infinity Laplacian.}
This highly nonlinear and strongly degenerated operator was
discovered by Aronsson in the sixties \cite{Aro}. However, the study of boundary value problems for the infinity Laplacian started only in the early nineties, 
with the advent of viscosity solutions theory. 
Bhattacharya, DiBenedetto and Manfredi were the first to consider
the Dirichlet problem for infinity harmonic functions and  to prove the existence of a solution in the viscosity sense \cite{BDM}; shortly afterwards, a fundamental contribution came  
by Jensen \cite{Jen}, who proved the validity of the comparison principle for the infinity Laplacian, yielding the uniqueness of solutions (see also \cite{ArCrJu, BaJeWa}).  
The last decade has seen a renewed and increasing interest around the infinity Laplacian, also due to its connections with differential games. 
With no attempt of completeness, among the topics under investigation in this growing field let us mention: 
inhomogeneous equations \cite{BhMo, LuWang}, regularity of solutions \cite{CFe, EvSav, EvSm, Lind, SWZ},  
ground states \cite{JLM,HSY,Yu,CF7}, overdetermined problems \cite{CFc, CFd, CFf}, tug-of-war games \cite{KoSe, PSSW}. 
In this scenario, the study of free boundary problems involving the infinity Laplacian seems to be rather at its early stage. 
To the best of our knowledge, only the following exterior version,
in the complement $\omega := \R^n\setminus\overline{\Omega}$ of an open bounded
convex set $\Omega$,
of Bernoulli problem has been considered in the literature (see \cite{ManPetSha}):
\[
\begin{cases}
\Dinf u = 0 & \text{ in } \omega^+(u):= \{x\in\omega:\ u(x) > 0\},
\\
u = 1 & \text{ on } \partial \omega,
\\ 
|\nabla u| = a(x)
& \text{ on } F(u)\,. 
\end{cases} 
\]
In particular, when $\Omega$ is a regular convex set  and $a (x) \equiv \lambda$, the situation looks relatively simple: a unique explicit solution exists, given by $1 -  \frac{1}{\lambda}\, \dist(x, \partial\omega)$. 
It satisfies the condition $|\nabla u | = \lambda$  in the classical sense 
along its free boundary, which 
in this case is of class $C^1$.
Further,  such solution can be identified with the pointwise limit, as $p \to + \infty$, of the unique solutions $u _p$ to the 
exterior Bernoulli problem for the $p$-Laplacian. 

On the variational side, let us mention that the asymptotics as $p \to + \infty$ of integral energies  associated with the exterior $p$-Bernoulli problem
(loosely speaking, functionals of the type  \eqref{f:Jpi} with $\Omega$ replaced by its complement)
has been studied in \cite{KawSha}. 
In a somewhat close spirit, the limiting behaviour as $p \to + \infty$ 
of the solutions of the minimization problems 
for the $p$-Dirichlet integral with a positive boundary datum and a constraint on the volume of the support, has been studied in \cite{RossiTeix}. 
Still, in the theme of free boundary problems for the infinity Laplacian, see also \cite{RUT, RossiWang,TU17}.  

\subsection{Notion of solution.} A delicate point before starting the analysis of problem \ref{f:P} is to establish what is meant by a solution. 
Clearly the PDE has to be understood in the viscosity sense. 
Going further we point out that, 
contrarily to the case of the exterior problem mentioned above,  for solutions to problem~\ref{f:P} the free boundary need not be globally $C^1$. Consequently,  a solution is not expected to be  differentiable up to the boundary (see \cite{Hong, Hong2}), so that also the free boundary condition cannot be interpreted in a pointwise, classical way. Thus, even at the boundary, a viscosity  interpretation seems to be
the most convenient one in order to manage both existence and uniqueness questions.  
More precisely, throughout the paper 
we interpret solutions to~\ref{f:P}
to mean viscosity solutions defined in the next definition,
introduced by De Silva \cite[Definitions~2.2 and~2.3]{DeSilva2011}, 
and has been adopted in several subsequent works 
(see for instance \cite{DFS17, LR18, LT15}). 

If $u,v\colon\Omega\to\R$ are two functions and
$x\in\Omega$, by
$u \prec_x v$
we mean that $u(x) = v(x)$ and $u(y) \leq v(y)$ 
in a neighborhood of $x$.
Moreover, in the following definition
for any test function $\varphi$ of class $C^2$, 
we set $\varphi^+ := \max \{ \varphi, 0 \}$. 

\begin{definition}\label{d:desilva}
	A non-negative function $u\in C(\overline{\Omega})$ is a viscosity solution to \ref{f:P} if
	
	\smallskip
	\begin{itemize}
		\item[(a)]
		$u$ is infinity harmonic at every $x \in \Omega ^+ (u)$, i.e., for any test function $\varphi$ of class $C^2$, 
		\begin{itemize}
\item[(a1)] if $u\prec_x \varphi$, then $-\Delta_\infty\varphi(x)\leq 0$; 

\smallskip
\item[(a2)] if $\varphi \prec_x u$, then $-\Delta_\infty\varphi(x)\geq 0$;
\end{itemize}

\smallskip
		\item[(b)] 
		the Dirichlet condition $u = 1$ holds pointwise on $\partial\Omega$;
		
		\smallskip
		\item[(c)]
		the free boundary condition holds at every $y \in F (u)$, i.e., for any test function $\varphi$ of class $C^2$ with $\nabla \varphi (y) \neq 0$, 
		\begin{itemize}
			\item[(c1)]
			if  $\varphi^+ \prec_{y} u$, 
			then $|\nabla\varphi(y)| \leq \lambda$;
			
			\smallskip
			\item[(c2)]
			if $u \prec_{y} \varphi^+$, 
			then $|\nabla\varphi(y)| \geq \lambda$.
		\end{itemize} 
	\end{itemize}
\end{definition}
It is clear from the definition that  $u=0$ on 
$\Omega\setminus\Omega^+(u)$, hence, in particular, on $F(u)$.

We point out that a solution in the sense of Definition~\ref{d:desilva},
is also a solution in the sense 
proposed by Caffarelli in
\cite[Definition~1]{Caff-H1}  (see also \cite{Caff-H2, Caff-H3}). 
The converse is a priori not true, because a touching ball as in Caffarelli's definition does not exist necessarily at all points of the free boundary. 
Some of our results (e.g., Proposition~\ref{p:gradbound} and 
Proposition~\ref{c:gradbound}) remain true if solutions are intended in the sense of \cite{Caff-H1}. 
However, Definition~\ref{d:desilva}  \`a la De Silva seems to be the one which allow us to deal in an optimal way with the existence question 
(in particular, in the proof of Theorem \ref{t:B} (b)).

\subsection{Synopsis of the results.} We carry over a detailed analysis of problem \ref{f:P}, 
which covers existence, uniqueness, and characterization of solutions, their regularity, and their relationship with the solutions to the interior Bernoulli problem for the $p$-Laplacian. We postpone  to a companion paper \cite{CF9} the study of the variational problem which is naturally associated with~\ref{f:P}, namely the minimization of the supremal functional 
\[
J _\lambda (u) := \|\nabla u\|_\infty + \lambda | \{ u > 0 \} |
\]  
over the space
of functions $u\in C(\overline{\Omega}) \cap W^{1,\infty}(\Omega)$  
which are equal to $1$ on $\partial \Omega$.

\smallskip
$\bullet$ {\it Existence.} By analogy with the case of the $p$-Laplacian, we define the $\infty$-Bernoulli constant of $\Omega$ as 
\begin{equation}\label{f:l}
\lambda _{\Omega, \infty}  := \inf  \Big \{ \lambda >0 :\  \text{ \ref{f:P} admits a non-constant solution }\Big \} \,.
 \end{equation} 
Then we identify $\lambda _{\Omega, \infty}$  with the reciprocal of the inradius $\inradius$ of $\Omega$. Indeed, for $\l < 1/\inradius$, problem \ref{f:P} does not admit any non-constant solution (Theorem \ref{t:B} (b)). The proof is based on a gradient estimate obtained via the gradient flow for infinity harmonic functions (Proposition~\ref{p:gradbound}).  
Conversely, for $\l \geq 1/\inradius$, we get existence. 
More precisely, it is convenient to distinguish between genuine
and non-genuine solutions, according to whether the set $\{ u = 0 \}$ has positive Lebesgue measure or not.  
For any $\l \geq 1/\inradius$, it is easily seen that problem~\ref{f:P} admits 
many non-genuine solutions, given by the infinity harmonic potentials
(see Definition~\ref{d:inftypot} below) 
of suitable compact subsets with empty interior contained in  the set of points $x \in \Omega$ with 
$\dist(x, \partial \Omega)\geq 1/\l $ (Proposition~\ref{p:trivial}). 
So the interesting feature is the existence of a genuine solution: 
if $\lambda > 1/\inradius$,
we show that it is given precisely by the infinity harmonic potential $w_{1/\l}$ of the set $\overline{\Omega_{1/\l}}$, being $\Omega_{1/\l}$ the set of points 
$x \in \Omega$ such that $\dist(x, \partial \Omega) > 1/\l$ (Theorem~\ref{t:B} (a)). This is obtained by constructing suitable upper and lower bounds for $w_{1/\l}$, and 
taking advantage of the simple behaviour of infinity harmonic potentials along rays of the distance function (see Section~\ref{ss:potentials}).

\smallskip
$\bullet$ {\it Uniqueness.} 
For $\l > 1/\inradius$, we obtain uniqueness of genuine solutions under two assumptions on the set $\Omega _{1/\l}$: connectedness and  ``open regularity" (Theorem \ref{t:unique}); moreover, we show that these assumptions are sharp (Examples \ref{e:nonconn} and \ref{e:nonreg}). 
It turns out that they are satisfied for example when $\Omega$ is convex. Remarkably, 
such uniqueness result on convex domains distinguishes the case of the $\infty$-Laplacian from the case of the $p$-Laplacian, when 
we have multiplicity of solutions also in case of the ball.

\smallskip
$\bullet$ {\it Characterization of solutions.} For $\l \geq 1/\inradius$, we show that $u$ is a solution to \ref{f:P} if and only if it is the 
infinity harmonic potential of  a set $K$ belonging to a suitable family of compact subsets of $\overline{\Omega _{1/\l}}$. 
This result (Theorem~\ref{t:solPl2}) gives a complete picture of 
solutions to~\ref{f:P} in case $\Omega$ is an arbitrary domain. 

\smallskip
$\bullet$ {\it Regularity.}  As a by-product of the results described so far, combined with well-known facts about the regularity of 
infinity harmonic functions, 
we obtain that, for $\l \geq 1/\inradius$,  any genuine solution is everywhere differentiable in $\Omega^+(u)$
(and $C^{1, \alpha}$ in dimension $n=2$).
Furthermore, the free boundary essentially shares the same regularity
properties of the level set $\{\dist(x, \partial \Omega) = 1/\lambda\}$ of the distance function.
More precisely, if we denote by $\Sigma(\Omega) $ the cut locus of $\Omega$
(i.e., the closure of the set of points where the distance from $\partial \Omega$ is not differentiable), then $F(u) \setminus \Sigma(\Omega) $ is locally $C^{1,1}$.
As a particular case, if $\lambda > 1/ \dist(\partial\Omega, \Sigma(\Omega) )$,
then $F(u)$ is of class $C^{1,1}$ and, if in addition
$\partial\Omega$ is of class $C^{k,\alpha}$ for some $k\geq 2$,
then $F(u)$ is of class $C^{k,\alpha}$
(see e.g.\ \cite[Theorem~6.10]{CMf}).

\smallskip
$\bullet$ {\it Relationship with the $p$-Bernoulli problem.} We show that, if $\Omega$ is convex and regular, both the $p$-Bernoulli constants $\l _{\Omega, p}$ and $\L _{\Omega, p}$ 
defined as in Section \ref{s:ip} above converge to $\l _{\Omega, \infty} = 1/R_\Omega$ in the limit as $p \to + \infty$ (Corollary \ref{c:convl}).
Moreover, if $u _p$ are solutions to the interior $p$-Bernoulli problem, we prove that they converge uniformly to the solution 
to problem \ref{f:P} provided
such solution is unique, and provided $u _p$ are {\it variational} solutions, namely they are minimizers 
of functionals~\eqref{f:Jpi} over $W^{1,p}_1 (\Omega)$ (Theorem~\ref{t:convp}). 
\subsection{Open problems} Let us conclude this Introduction by addressing some among the many interesting questions related to the results contained in this paper: 
\begin{itemize}
\item[(i)] Is it possible to extend at least some of our results to the case of  non-constant boundary data?
\item[(ii)] Does the unique solution to problem~\ref{f:P} on a convex domain have convex level sets?
\item[(iii)] In cases when there are multiple genuine solutions, does it exist a {\it minimal} genuine solution, and how can it be characterized?
\item[(iv)] When the solution 
to problem \ref{f:P} is not unique, is it still true that 
the variational solutions $u_p$ to the interior $p$-Bernoulli problem converge in the limit as $p \to + \infty$, and what is their limit?
\end{itemize}

\bigskip

\section{Some preliminary results }

In this section we collect some material
which will be useful throughout the paper. 
To be self-contained, we start by giving a quick recall of
some basic facts about infinity harmonic functions,
for which we refer to \cite{ArCrJu,Cran,CEG}.  

Then we establish some general properties of (non-constant) solutions to \ref{f:P}
and  of infinity harmonic potentials, which will play a crucial role in the sequel. 

\smallskip
Let us firstly introduce some notation. 
We shall write for brevity
$d(x) := \dist(x, \partial\Omega)$, $x\in\overline{\Omega}$. 
Moreover, we denote by 
$\inradius := \max_{\overline{\Omega}} d$ the inradius of $\Omega$, 
and for any $r \in [0, \inradius]$, we set
\begin{gather*}
\Omega_r = \Oma{r} := \{x\in\Omega:\ d(x) > r\}\,,
 \\
\Omc{r} := \{x\in\Omega:\ d(x) \geq r\}\,,
\\
D_r:= \Omega \setminus \overline {\Omega _r}\,.
\end{gather*}
For every $x\in\overline{\Omega}$ we denote by
\begin{equation}\label{f:Pi}
\Pi _{\partial \Omega} (x) := \big \{ z \in \partial \Omega\, : \, d (x) = |z-x| \big \}
\end{equation}
the set of the closest points (or projections) of $x$ on $\partial\Omega$.

\subsection{About infinity harmonic functions.} 
A function $u\in C (\Omega)$ is called \textit{infinity subharmonic} (resp.\ \textit{infinity superharmonic}) if it satisfies 
condition (a1) (resp. (a2)) in Definition~\ref{d:desilva}. It is called infinity harmonic if it is both infinity subharmonic and superharmonic. 

An infinity harmonic function on $\Omega$ is differentiable at every point $x \in \Omega$ in any space dimension, and of class $C ^ {1, \alpha}(\Omega)$ 
in dimension $n =2$ \cite{EvSm, EvSav, Sav}. 

The following conditions are equivalent:

\begin{itemize}
\item[(i)]  $u$ is infinity harmonic in $\Omega$; 

\smallskip
\item[(ii)]  
$u$ has the {\it absolutely minimizing Lipschitz} property,
which means that $u$ is locally Lipschitz in $\Omega$
and, for every open set $\omega \Subset \Omega$
and every $v \in C (\overline \omega)$, with $v = u$ on $\partial \omega$,
$\|\nabla u\|_{L^\infty(\omega)} \leq \|\nabla v\|_{L^\infty(\omega)}$.
The space of functions $u$ having this property
is denoted by $AML ( \Omega)$;

\smallskip
\item[(iii)] the functions $w=u$ and $w=-u$ {\it enjoy comparison with cones from above} in $\Omega$, which means that, 
for every open set $\omega \Subset \Omega$ and for every 
$a, b \in \R$ and $x_0 \in \R ^n$, it holds
\[
w (x) \leq C (x):= a + b |x-x_0|\,, \  \forall x \in \partial (\omega \setminus \{ x_0\}) 
\quad \Longrightarrow \quad 
w (x) \leq C ( x)\,, \ \forall x \in \omega \,.
\]
\end{itemize}

Let $u$ be infinity harmonic in $\Omega$, and let
$\overline{B}_r(x) \subset\Omega$.
Then 
\begin{equation}
\label{f:maxmin}
\max_{y\in \overline{B}_r(x)} u(y) = \max_{y\in \partial{B}_r(x)} u(y),
\qquad
\min_{y\in \overline{B}_r(x)} u(y) = \min_{y\in \partial{B}_r(x)} u(y),
\end{equation}
and the following relations hold:
\begin{gather}
|\nabla u(x)| \leq \max_{y\in \overline{B}_r(x) } \frac{u(y)-u(x)}{r}
= \max_{y\in \partial{B}_r(x) } \frac{u(y)-u(x)}{r}, \label{f:cr1}\\
|\nabla u(x)| \leq -\min_{y\in \overline{B}_r(x) } \frac{u(y)-u(x)}{r}
= -\min_{y\in \partial{B}_r(x) } \frac{u(y)-u(x)}{r} \label{f:cr2}
\end{gather} 
(see \cite[Lemma~4.3]{Cran}).
Moreover, if the maximum and minimum at the right--hand side of \eqref{f:cr1}, \eqref{f:cr2}
are attained respectively at $p,q\in\partial B_r(x)$, i.e.\ if
\[
p,q\in\partial B_r(x):
\quad
u(p) = \max_{y\in \partial{B}_r(x)} u(y),
\quad
u(q) = \min_{y\in \partial{B}_r(x)} u(y),
\]
then the following
\textsl{increasing slope estimates} hold:
\begin{equation}
\label{f:ise}
|\nabla u(x)| \leq |\nabla u(p)|,
\quad
|\nabla u(x)| \leq |\nabla u(q)|
\end{equation}
(see \cite[Proposition~6.2]{Cran}).

\subsection{Properties of solutions to \ref{f:P}}

Observe that,
if $u$ is a strictly positive solution to~\ref{f:P},
then by uniqueness $u\equiv 1$.
Hence, any non-constant solution to~\ref{f:P} must vanish
at some point of $\Omega$, i.e., 
$F(u)\neq\emptyset$.

\begin{proposition}[gradient estimate]\label{p:gradbound}
	Let $u\in C(\overline{\Omega})$ be a solution
	to~\ref{f:P}.
	Then $|\nabla u(x)| \leq \lambda$ for every $x\in \Omega^+(u)$.
\end{proposition}

\begin{proof}
If $u$ is a constant solution then the result is trivial.
Let $u$ be a non-constant solution.
	Let $x_0\in\Omega^+(u)$ and let us prove that $|\nabla u(x_0)| \leq \lambda$.
	Since the statement is trivial if $\nabla u (x_0) = 0$, let us assume that $\nabla u (x_0) \neq 0$. 
	In this case, we claim that there exists a finite family $x_0, x_1, \ldots, x_N$ of points with the following
	properties:
	\begin{gather}
		x_0, \ldots, x_{N-1} \in \{ u \leq u (x_0)  \} \cap \Omega ^+ (u),
		\quad 
		x_N \in F(u),
		\label{f:fpt1}
		\\
		|\nabla u(x_j)| \geq |\nabla u(x_{j-1})|
		\ \forall j = 1, \ldots, N-1,
		\
		u(x_{N-1}) \geq \dist(x_{N-1}, F(u)) |\nabla u(x_{N-1})|.
		\label{f:fpt2}
	\end{gather}

Since $u (x_0) <1$ and $u$ is continuous, the sub-level $\mathcal C:= \{ u \leq u (x_0)  \}$ is a compact  subset of $\Omega$. Hence we can find $\rho >0$ such that $\mathcal C \subset \Omega _\rho$. 
	Then we fix $r \in (0, \rho)$ and we proceed as follows.
	
	Assume we are given $x_{j-1}\in \{ u \leq u (x_0)  \} \cap \Omega ^+ (u)$, and let us construct
the point $x_j$.
	
	If $\overline{B}_r(x_{j-1}) \subset \Omega^+(u)$,
	then we let $x_j\in\overline{B}_r(x_{j-1})$ be such that
	\[
	u(x_j) = \min_{y\in \overline{B}_r(x_{j-1})} u(y).
	\]

	By definition, we have immediately $u(x_j) \leq u(x_{j-1})$, so that $x_j\in \mathcal C \cap \Omega ^ + (u)$.
Moreover, 
since $u$ is infinity-harmonic in $\Omega^+(u)$, 
by \eqref{f:maxmin} and \eqref{f:ise}
it turns out that $x_j\in\partial B_r(x_{j-1})$ and 
	$|\nabla u(x_j)| \geq |\nabla u(x_{j-1})|$.
	
	If $\overline{B}_r(x_{j-1})$ is not contained in $\Omega^+(u)$, 
	by our choice of $r$ we have necessarily 	$\overline{B}_r(x_{j-1})\cap F(u) \neq \emptyset$. (Indeed, 
	since $x _{j-1} \in \mathcal C \subset \Omega _\rho$ and $r \in (0, \rho)$, 
	we have 
	$\overline{B}_r(x_{j-1}) \cap\partial\Omega = \emptyset$.) 
In this case,  we set $N = j$, ending the construction, and we let $x_N\in F(u)$ be a closest point of $x_{N-1}$ to $F(u)$.
	Setting 
$\delta := \dist(x_{N-1}, F(u)) = |x_N - x_{N-1}|$ and taking into account
	$u(x_N) = 0 = \min_{y \in \overline{B}_{\delta}(x_{N-1})} u(y)$,
	by \eqref{f:cr2}
	we obtain 
		$$|\nabla u (x_{N-1})| \leq - \min _{y\in \overline{B}_\delta(x_{N-1})} \frac{u (y) - u (x_{N-1})}{\delta} =  \frac{u (x_{N-1}) }{\delta}\,. $$ 

	It remains to show that our construction always stops in a finite number of steps.  
	Specifically, for every $j = 1, \dots , N -1$, applying again~\eqref{f:cr2},
	we obtain	
\[
|\nabla u (x_{j-1})| \leq - \min _{y\in \overline{B}_r(x_{j-1})} \frac{u (x) - u (x_{j-1})}{r} =  \frac{u (x_{j-1}) - u (x_{j})}{r} ;
\] 
	hence	 
\[
	u(x_j) \leq u(x_{j-1}) - r |\nabla u(x_{j-1})| 
	\leq u(x_{j-1}) - r |\nabla u(x_{0})|\,,
	\]
	 so that in a finite number of steps we arrive at $F (u)$ thanks to the assumption $\nabla u (x_0) \neq 0$.

	\smallskip
	Now, let us consider
	the open ball $ B_\delta(x_{N-1}) \subset \Omega^+(u)$. By comparison with cones \cite[Theorem 3.1]{CEG}, we have 
	\begin{equation}\label{f:cones}
	u(x) \geq 
\varphi(x):=
u(x_{N-1}) \left(1 - \frac{1}{\delta}|x-x_{N-1}|\right) \qquad \forall x \in  B_\delta(x_{N-1})\,.
	\end{equation}
Since $x_N\in\partial B_\delta(x_{N-1})$, by~\eqref{f:cones} we have that
\[
\varphi^+ \prec_{x_N} u\, , \qquad |\nabla\varphi(x_N)|= \frac{u(x_{N-1}) }{ \delta} (\neq 0) \,.
\] 
Then, by applying first Definition~\ref{d:desilva}(c1) and then the inequalities \eqref{f:fpt2}, we finally get	
	\[
	\lambda \geq \frac{u(x_{N-1})}{\delta} \geq 
	|\nabla u(x_{N-1})| \geq
	|\nabla u(x_0)|,
	\]
	and the proof is completed. 
\end{proof}

\begin{proposition}[free boundary location]\label{c:gradbound}
	Let $u\in C(\overline{\Omega})$ be a non-constant solution
	to \ref{f:P}. 
	Then $\dist(F(u), \partial\Omega) \geq \frac{1}{\lambda}$ 
	(or,
	equivalently,
	$\{u = 0\} \subseteq \Omc{\frac{1}{\lambda}}$).
	If, in addition, $\interior\{u=0\}\neq\emptyset$, then $\dist(F(u), \partial\Omega) = \frac{1}{\lambda}$.
\end{proposition}

\begin{proof}
	Let $x\in F(u)$ and let 
$y\in\Pi _{\partial \Omega}(x)\subset\partial\Omega$ be a closest point to $\partial\Omega$.
	If $]y,x[ \cap F(u) \neq \emptyset$, let $x_0\in ]y,x[ \cap F(u)$ be the nearest point
	of $]y,x[ \cap F(u)$ to $\partial\Omega$, otherwise let $x_0 := x$.
	By Proposition~\ref{p:gradbound}, we have 
	\[
	1 = u(y) - u(x_0) \leq \lambda\, d(x_0),
	\]
	hence
	\[
	d(x) \geq d(x_0)\geq
	\frac{1}{\lambda},
	\]
	i.e.\ $x\in \Omc{\frac{1}{\lambda}}$.
	Hence, $F(u) \subseteq \Omc{\frac{1}{\lambda}}$, i.e.\ 
	$\dist( F(u), \partial\Omega) \geq \frac{1}{\lambda}$.
	
	Let us prove that, if $\interior\{u=0\}\neq\emptyset$,
	then also the opposite inequality holds.
	Let $r := \dist(\partial\Omega, F(u))$.
	The function $v(x) := \frac{1}{r}\, \dist (x, F(u))$ is infinity superharmonic in $\Omega^+(u)$
	(see {\it e.g.}\ \cite[p.~212]{JuuLinMan}), and satisfies
	$v=0$ on $F(u)$ and $v\geq 1$ on $\partial\Omega$.
	Hence, by the comparison principle  for infinity harmonic functions \cite[Theorem 2.22]{Jen},
	we have that $v \geq u$ in $\Omega^+(u)$.
	Since $\interior\{u=0\}\neq\emptyset$, there exists a ball $B=B_\rho(y)\subset\Omega\setminus\Omega^+(u)$
	that is tangent to $F(u)$ at some point $x_0\in F(u)$.
	Hence, 
	\[
	u(x) \leq v(x) \leq \frac{|x-x_0|}{r}\leq \frac{|x-y|-\rho}{r} =: \varphi(x),
	\qquad x\in \Omega^+(u), 
	\]	
so that
\[
u \prec_{x_0} \varphi ^+\, , \qquad |\nabla\varphi(x_0)|=  \frac{1}{r} (\neq 0) \,,
\] 
and by Definition~\ref{d:desilva}(c2) we conclude that $1/r \geq \lambda$. 
\end{proof}

\bigskip

\subsection{Properties of infinity--harmonic potentials.}
\label{ss:potentials}

\begin{definition}\label{d:inftypot} 
Given a non-empty compact set  $K\subset\Omega$, 
the infinity--harmonic potential of $K$ relative to $\Omega$ 
is the unique viscosity solution $w_K$ to the problem
\begin{equation}\label{f:pot}
\begin{cases}
-\Dinf w_K = 0, & \text{in}\ \Omega\setminus K,\\
w_K  = 1, &\text{on}\ \partial\Omega,\\
w_K  = 0, &\text{on}\  K.
\end{cases}
\end{equation}
\end{definition}

\begin{remark}\label{r:cc}  Since $\Omega \setminus K$ may be disconnected, some words to explain the well-posedness of the above definition are in order. Let us write the open set $\Omega\setminus K$ as the union of its connected components
$\{A^\alpha:\ \alpha\in I\}$. For every $\alpha\in I$, 
we have that $\partial A^\alpha \subseteq\partial\Omega \cup K$,
and the function $f^\alpha\colon\partial A^\alpha\to\R$ defined by
\begin{equation}\label{f:fa}
f^\alpha :=
\begin{cases}
1, &\text{on}\ \partial A^\alpha \cap \partial\Omega,\\
0, &\text{on}\  \partial A^\alpha \cap K,
\end{cases}
\end{equation}
is continuous on $\partial A^\alpha$ (being constant on each connected component
of $\partial A^\alpha$). Therefore, for every $\alpha \in I$, there exists a unique solution $w^\alpha \in C(\overline{A^\alpha})$
to the Dirichlet problem
\[
\begin{cases}
-\Dinf w^\alpha = 0, & \text{in}\ A^\alpha,\\
w^\alpha = f^\alpha, & \text{on}\ \partial A^\alpha
\end{cases}
\]
 (see \cite[Theorems 3.1 and 6.1]{ArCrJu}).
 Consequently,  problem
\eqref{f:pot} admits a unique solution, which is precisely 
the function $w_K\in C(\overline\Omega)$ defined by
$w_K = w^\alpha$ on $\overline{A^\alpha}$, $\alpha\in I$. 
\end{remark}

\begin{remark} It is clear from Definition \ref{d:inftypot} that the set $K$ is contained in   $\{w_K = 0\}$. We point out that the inclusion may be strict.
For instance, this happens when  $\Omega = B_2(0)$ and $K = \partial B_1(0)$: in this case, $K$ is strictly contained in $\{w_K = 0\} = \overline{B}_1$.  
\end{remark}

In general, given a non-empty compact set  $K\subset\Omega$, necessary and sufficient conditions for the equality $K=\{w_K = 0\}$ can be given by looking at 
the connected components $A ^ \alpha$ of $\Omega \setminus K$ introduced in Remark \ref{r:cc}. 
Letting 
\begin{equation}\label{f:I0}
I _0:=\{\alpha\in I:\ \partial A ^ \alpha \subseteq K\}
\end{equation}
we have

\begin{lemma}\label{l:clarify}
For a given non-empty compact set  $K\subset\Omega$,	the following properties are equivalent:
	\begin{itemize}
	
	\item[1)] $\{w_K = 0\} = K$; 
	
		\item[2)]  the set $I_0$ defined in \eqref{f:I0} is empty;

		\item[3)]
		every point $x\in\Omega\setminus K$ can be joined to $\partial \Omega$ through a path in  $\overline \Omega \setminus K$. 
	\end{itemize}
\end{lemma} 

\begin{proof} 
The equivalence between 2) and 3) follows immediately from the fact that the points 
 in $\overline \Omega \setminus K$  which can be joined to $\partial \Omega$ through a path in  $\overline \Omega \setminus K$ are precisely the points in 
$\overline \Omega \setminus K$ which belong to some 
set $A ^ \alpha$ with $\alpha \in I \setminus I _0$.

The fact that each of these conditions is equivalent to 1) follows by observing that

\begin{equation}\label{f:0set}
\{w_K = 0\} = K \cup \bigcup_{\alpha\in I_0} A^\alpha\, , 
\qquad
\{w_K > 0\}= \bigcup_{\alpha\in I\setminus I_0} A^\alpha\,.
\end{equation}

Indeed, it is clear that $f^\alpha \equiv 0$ for $\alpha \in I _0$. On the other hand,  for 
every $x_0$ belonging to a set $A ^ \alpha$ with $\alpha \in I \setminus I _0$, $w_K (x_0)$ is strictly positive, because
such a point
$x_0$ can be joined to $\partial \Omega$ through a path in  $\overline \Omega \setminus K$.  Consequently, 
the value $w_K (x_0)$ can be estimated from below by a positive constant according to the next result (which is essentially taken from \cite[Lemma 3.2]{Bh2002}).
\end{proof}

\begin{proposition}[Harnack inequality]\label{p:harnack}
	Let $K\subset\Omega$ be a non-empty compact set, and let $w_K$ be its infinity--harmonic potential  relative to $\Omega$. 
	Let $x_0\in A ^ \alpha$, with $\alpha \in I \setminus I _0$, and let 
	$\gamma$ be a path in  $\overline{\Omega}\setminus K$ connecting $x_0$ to $\partial\Omega$.
	Then 
	\begin{equation}\label{f:hk}
	w_K(x_0) \geq e^{-L/\delta},
	\end{equation}
	where $L$ is the length of $\gamma$, and $\delta$ is the distance from $\gamma$ to $K$.
\end{proposition}

\begin{proof}

By possibly taking a slightly larger value of $\delta$ (but smaller than $\dist(\gamma, K)$),
it is not restrictive to assume that $\gamma$ is a polygonal curve.
Moreover, for $m\in\N$ large enough, we can assume that the polygonal has exactly $m+1$
vertices $x_0, x_1, \ldots, x_m = y$ with
$|x_j - x_{j-1}| = L/m$ for every $j=1,\ldots,m$.
By possibly moving a bit the point $y$ (shortening the curve),
we can also assume that $y$ is a 
closest point in $\partial\Omega$ from $x_{m-1}$.
Since $x_{j-1}\in B_\delta(x_{j}) \subset \Omega\setminus K$ for every $j = 1, \ldots, m$,
by comparison with cones,  we have 
\[
w_K(x_{j-1}) \geq w_K(x_j) \left(1-\frac{|x_j - x_{j-1}|}{\delta}\right)
= w_K(x_j) \left(1-\frac{L}{m\delta}\right),
\]	
so that
\[
w_K(x_0) \geq w_K(y) \left(1-\frac{L}{m\delta}\right)^m.
\]
Since $w_K(y) = 1$ and $m$ can be taken arbitrarily large, we finally get \eqref{f:hk}.
\end{proof}

\bigskip
We conclude with  a useful characterization of the infinity harmonic potential $w_K$ along rays connecting $K$ with $\partial \Omega$:

\begin{proposition}[potential along rays]\label{p:pot}
	Let $K\subset\Omega$ be a non-empty compact set, and let $w_K$ be the infinity--harmonic potential of $K$ relative to $\Omega$.
	If $y\in\partial\Omega$ and $z\in K$ are two points
	such that $|y-z| = \dist(\partial\Omega, K)$, then $w_K$ is affine on the segment $[y,z]$.
\end{proposition}

\begin{proof} 
	Set $R:= \dist(\partial\Omega, K)$. Since $w_K$ enjoys comparison with cones from below, we have 
	\[
	w_K(x) \geq f(x) := 1 - \frac{|x-y|}{R},
	\qquad
	\forall x \in B_R(y) \cap \Omega.
	\]
	On the other hand, 
	the function $g(x) := \frac{1}{R}\, \dist(x, K)$ is infinity superharmonic
	in $\Omega\setminus K$, with 
	$g=0$ on $K$ and $g\geq 1$ on $\partial\Omega$, hence
	$g\geq w_K$ by the comparison principle for infinity harmonic functions. 
	Since $f = g$ on the segment $[y,z]$,
	the statement follows.
\end{proof}

\section{Existence}

We start the analysis of existence of solutions to problem \ref{f:P} 
by observing that, for any $\lambda \geq 1/\inradius$,  it admits many  solutions whose zero level set is Lebesgue negligible.  
Inspired by the results of the previous section, they are found among
 infinity harmonic potentials $w_K$ of suitably chosen compact
sets $K$ contained in  
$\Omc{\frac{1}{\lambda}}$. 
Recall that the zero set of $w_K$ can be characterized as in \eqref{f:0set}; in particular, by Lemma~\ref{l:clarify},
we have that $\{w_K = 0\} = K$ if and only if the set $I_0$ defined in~\eqref{f:I0} is empty.

\begin{proposition}\label{p:trivial}
	Let $\lambda \geq 1/\inradius$, and let $K \subseteq \Omc{\frac{1}{\lambda}}$ be a
	non-empty compact set. 
	Assume that
	\begin{equation}\label{f:trivial}\interior (K) = \emptyset \qquad \text{  and  } \qquad I _0 = \emptyset \,.\end{equation}
	Then the infinity--harmonic potential $w_K$ of $K$ relative to $\Omega$
 is a solution to \ref{f:P}.
\end{proposition}

\begin{proof}
	By \eqref{f:trivial} and \eqref{f:0set}, the set 
	$\{w_K = 0\}$  agrees with $K$ and has empty interior,
	so that $\{w_K = 0\} = K = F(w_K)$.
	Thus, we have to show that the free boundary condition in Definition~\ref{d:desilva}
	is satisfied at every point $x_0 \in K$.
	
	Since $K\subseteq \Omc{\frac{1}{\lambda}}$, by comparison with cones we have that
	$w_K(x) \leq \lambda |x-x_0|$ for every $x\in\Omega$.
	If $\varphi^+ \prec_{x_0} w_K$, then necessarily $|\nabla\varphi(x_0)| \leq \lambda$, 
	hence condition (c1) in Definition~\ref{d:desilva} is satisfied.
	
	If $w_K \prec_{x_0} \varphi^+$, 
	then $\varphi \geq 0$ in $\Omega$, because $\varphi^+ \geq w_K > 0$ in $\Omega\setminus \{w_K= 0 \}$
	and $\overline {\Omega \setminus \{w_K = 0 \}} = \overline \Omega$.
	Since $\varphi(x_0) = 0$, then $x_0$ is a minimum point for the regular function
	$\varphi$, hence we can conclude that $\nabla\varphi(x_0) = 0$,
	and also condition (c2) in Definition~\ref{d:desilva} is satisfied. 	
\end{proof}

Motivated by Proposition~\ref{p:trivial}, we give the following definition.

\begin{definition}[Genuine solutions]\label{d:triv}
	We say that a solution $u$ to \ref{f:P} is genuine
	if the set $\{u = 0\}$ has non-empty interior (and non-genuine otherwise). 
\end{definition}

\begin{remark}\label{r:inradius}
In the special case $\lambda = \frac{1}{\inradius}$, problem \ref{f:P} admits only non-genuine solutions. Indeed, 
we know from  Proposition \ref{c:gradbound} that, for every solution $u$ to problem \ref{f:P},  $F (u)$ is contained in  the high ridge $\{ d (x) = \inradius \}$ and hence the set $\{ u = 0 \}$ has necessarily empty interior. 
\end{remark}

We are now going to deal with the existence of genuine solutions  to \ref{f:P}, for 
$\lambda > \frac{1}{\inradius}$.
To that aim, we introduce two more definitions.

\begin{definition}\label{d:w} 
Given $r \in (0, \inradius)$, we define $w_r$ as the infinity harmonic potential of $\overline \Omega _r $ relative to $\Omega$, 
namely the unique solution to
\[
\begin{cases}
\Dinf w_r = 0 & \text{ in } D_r:= \Omega \setminus \overline \Omega _r
\\
w_r= 1 & \text{ on } \partial \Omega
\\ 
w_r = 0 & \text{ in }  \overline \Omega_{r}\,.
\end{cases} 
\]

\end{definition}

\begin{figure}
\includegraphics{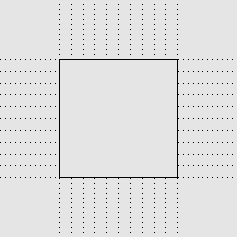}
	\caption{The sets $\Omega = (-2,2)\times (-2,2)$ (grey), $\partial\Omega_1$ (black),
	and $\widehat{D}_1$ (dotted)}\label{fig:hat}
\end{figure}

\begin{definition}\label{d:hatD} 
Given $r \in (0, \inradius]$, 
we set
\[
\widehat{D}_r := \bigcup_{y\in\partial\Omega_r}
\left\{
]y,z[:\ z\in\Pi_{\partial\Omega}(y)
\right\}\,,
\]
where $\Pi _{\partial \Omega}(y) \subset\partial\Omega$ is the set of projections of $y$
defined in~\eqref{f:Pi}.
\end{definition}

\begin{remark}\label{r:hatD} 
Notice that, by definition, $\widehat D _r$ is a subset of $D_r$, with possibly strict inclusion (think for instance to the case when $\Omega$ is a square, see Figure \ref{fig:hat}). 
\end{remark}

\begin{theorem}\label{t:B}
	(a) For every $\lambda >  \frac{1}{\inradius}$,
	the function $w_{\frac{1}{\lambda}}$ is a genuine solution to
	problem \ref{f:P}; moreover it satisfies the estimates
	\begin{equation}\label{f:cfdist}
		1- \lambda d(x) \leq  w_{\frac{1}{\lambda}} (x) 
		\leq \lambda \dist (x, \partial \Omega _{\frac{1}{\lambda}} )  \quad \text{ in }  \overline D_{\frac{1}{\lambda}} \,, 
		\text{ with equalities in } \widehat D_{\frac{1}{\lambda}}\,.
	\end{equation}
	(b) For every $\lambda \in \big( 0, \frac{1}{\inradius} \big)$,
	problem \ref{f:P} does not admit non-constant solutions.
\end{theorem}

\begin{proof}
Throughout the proof, since $\lambda$ is fixed, we set for brevity 
 $$w:=w _{\frac{1}{\l}}\,, \qquad D: = 
D_{\frac{1}{\lambda}}\,, \qquad  \widehat D: = 
\widehat D_{\frac{1}{\lambda}}\,.$$ 
Let us first show that $w$ satisfies the inequalities in  \eqref{f:cfdist}.
This can be proved using the same arguments of Proposition~\ref{p:pot}.
More precisely,
the function 
$v(x) := 1- \lambda d(x)$
is infinity subharmonic 
(since $d$ is infinity superharmonic), 
and  satisfies the equality $v = w$ on both $\partial \Omega$ and $\partial \Omega _{\frac{1}{\lambda}} $. By the comparison principle for infinity harmonic functions, it follows that 
$w  \geq v$ in $\overline D$.  

Similarly, the function $z(x) := \lambda \,  \dist ( \cdot, \partial \Omega_{\frac{1}{\lambda}})$ is infinity superharmonic, and satisfies 
$z= w = 0$ on $\partial \Omega _{\frac{1}{\lambda}} $, $w \leq z$ on $\partial \Omega$.  
Again by the comparison principle for infinity harmonic functions, we infer that $w \leq z $ in $\overline D$.

In order to obtain that the inequalities in \eqref{f:cfdist} hold as equalities in $\widehat D$, we firstly notice that 
$\|\nabla w\|_\infty = \lambda$.
Indeed, the inequality  $\|\nabla w\|_\infty \geq \lambda$ follows immediately from the estimate
\[
\|\nabla w\|_\infty \geq \sup  \Big \{ \frac{|w (x) - w (y)| }{|x-y| }  \ :\ x \in \partial \Omega, \ y \in \partial \Omega _{\frac{1}{\lambda}} \Big \}\,;
\] 
the converse one follows from the fact that $w$ has the AML property in $D$, which entails in particular
$\|\nabla w\|_\infty \leq  \|\nabla v\|_\infty = \l$.

Now assume by contradiction that the strict inequality $w   > v $ holds at some point $x \in \widehat D$. 
If  $x$ belongs to the segment $]y, z[$, with $y \in \partial \Omega _{\frac{1}{\lambda}}$ and $z \in \Pi _{\partial \Omega} (y)$,      
we have that
\[
\|\nabla w\|_\infty \geq \frac{|w  ( x) - w (y)| }{|x-y | } =
\frac{w ( x) }{|x-y | } 
> 
\frac {v  ( x) }{|x-y | } 
 = \frac{|v ( x) - v (y)| }{|x-y | }  = \l \,.
\] 
Here, in the last equality we have exploited the fact that 
$d(x) - d(y) = |x - y|$. 
Indeed, if $x \in ]y, z[ \subset \widehat D_r$, with $y \in \partial \Omega _r$ and 
$z \in \Pi_{\partial\Omega} (y)$, it holds that
$d(x) = r - |x-y|$ and $\dist(x, \partial \Omega_r) = r - |x-z|$, 
which implies in particular
\begin{equation}\label{f:coincidence}
 r - d(x)  = |x-y| = r - |x-z| =  \dist (x, \partial \Omega_r)\, .
 \end{equation}
We have thus contradicted the equality $\|\nabla w\|_\infty = \l$, 
and we conclude that $w (x) = v (x)$. 
Since, by~\eqref{f:coincidence}, $v ( x) = z (x) $ on $\widehat D$, the proof of \eqref{f:cfdist} is achieved.  
	
	\smallskip
(a) We are now in a position to prove that $w$ solves problem \ref{f:P},  which amounts to show that  it satisfies the free boundary condition (c) of Definition~\ref{d:desilva}
along the free boundary $F(w) = \partial\Omega_{\frac{1}{\lambda}} $.

Let $x_0\in \partial\Omega_{\frac{1}{\lambda}} $, 
let $\varphi^+ \prec_{x_0} w$, with $p := \nabla\varphi(x_0)\neq 0$.
By the upper bound inequality in \eqref{f:cfdist}, we have 
\[
\varphi(x) \leq w(x) \leq \lambda\, \dist (x, \partial \Omega_{\frac{1}{\lambda}} )
\qquad\forall x\in D,
\]
hence
\[
\varphi(x_0 + t p) \leq \lambda\, \dist (x_0 + t p,  \partial \Omega_{\frac{1}{\lambda}} ) \leq \lambda\, t\, |p|,
\qquad t>0\ \text{small}.
\]
Dividing by $t>0$ and taking the limit as $t\to 0^+$ we get
$|p|^2 \leq \lambda\, |p|$, hence $|p| \leq \lambda$,
so that (c1) holds.

Let us now consider condition (c2) at a point 
$x_0\in \partial\Omega_{\frac{1}{\lambda}}$.
Let $y\in\Pi_{\partial\Omega}(x_0)$. 
By \eqref{f:cfdist},  the function $w$  is affine with slope $\lambda$ on the segment  $]x_0, y[ \subset \widehat D$. 
If $\varphi$ is a test function as in condition (c2), 
setting $\nu := (y-x_0) / |y-x_0|$, 
we have that
\[
|\nabla\varphi(x_0)| \geq \lim_{t\to 0+} \frac{\varphi (x_0 + t\nu)}{t} \geq \lim_{t\to 0+} \frac{w(x_0 + t\nu)}{t}
= \lambda,
\]
and (c2) follows.

	\medskip
	(b)  We observe that, if $u$ is a non-constant solution to \ref{f:P} (for an arbitrary $\lambda > 0$), 
	it holds that
	\begin{equation}\label{f:supgrad}
	\sup_{x\in\Omega^+(u)} |\nabla u(x)| \geq 1/\inradius\,.
	\end{equation}
	Indeed, if we assume that $|\nabla u(x)| \leq L < 1/\inradius$
	for every $x\in \Omega^+(u)$, then we obtain
	\[
	u(x) \geq 1 - L\, d(x)
	\geq 1 - L\, \inradius > 0
	\qquad \forall x\in \overline{\Omega^+(u)},
	\]
	a contradiction.
	
Statement (b) is a direct consequence of Proposition~\ref{c:gradbound},
since $F(u)\neq\emptyset$ for non-constant solutions.	
\end{proof}

\section{Uniqueness}

Prior to starting the analysis of the  uniqueness of solutions for problem \ref{f:P},
we emphasize that one has to 
restrict attention to the class of genuine solutions  and to choose $\lambda > 1/\inradius$. 
Indeed, if these requirements are dropped, by applying the results of the previous section we readily get the following conclusions:

\begin{itemize}
\item[--]
For $\lambda > 1/\inradius$,   according to Proposition \ref{p:trivial} there exist infinitely many non-genuine solutions to \ref{f:P}, corresponding to the infinity harmonic potentials of any  compact set $K \subseteq \Omc{\frac{1}{\lambda}}$ satisfying \eqref{f:trivial}. 
 \medskip
\item[--]
For $\lambda = 1/\inradius$,  we know that all the solutions to \ref{f:P} are non-genuine
({\it cf.}\ Remark~\ref{r:inradius}). Moreover, it is easy to see that any compact set $K$ contained in  the high rigde of $\Omega$ satisfies \eqref{f:trivial}. Therefore, there exist either one or multiple 
non-constant solutions to \ref{f:P} respectively when the high ridge is a singleton or not. 
\end{itemize}

\medskip
We are thus led to formulate the question as: 

\medskip
\centerline{
\text{\it When 
uniqueness of genuine solutions to  \ref{f:P} occurs for $\lambda > 1/\inradius$? } 
}
 
 \medskip
Our answer is given in the statement below.

\begin{theorem}[Uniqueness of genuine solutions]\label{t:unique}
	Let $\lambda > 1/\inradius$. Assume that 
		\begin{itemize}
	\item[(H1)]  
	$\Omega _{\frac{1}{\lambda}}$  is connected;
	\smallskip
	\item[(H2)]   
	$\overline {\Omega _{\frac{1}{\lambda} }}=\Omc{\frac 1 \lambda}$. 
	\end{itemize}
		Then $w_{\frac 1 \lambda}$ is the unique genuine solution to \ref{f:P}.
\end{theorem}

\begin{corollary}\label{c:convex}
Assume  $\Omega$ is  convex. For every  $\lambda > 1/\inradius$, $w_{\frac 1 \lambda}$ is the unique genuine solution to problem \ref{f:P}.
\end{corollary}

\begin{remark}{\it (About the connectness assumption (H1)).} 
When $\Omega$ is convex, assumption (H1) is satisfied because also $\Omega _r$ is convex for every $r\in [0, \inradius)$. 
For general $\Omega$,  (H1) is satisfied if 
$\frac{1}{\lambda}< \dist(\partial\Omega, \Sigma(\Omega))$, $\Sigma (\Omega)$ being the cut locus of $\Omega$, namely
the closure of the set of points where the distance from $\partial \Omega$ is not differentiable. 
Indeed,   if $r < \dist(\partial\Omega, \Sigma(\Omega))$,
then $\Sigma(\Omega)\subset\Omega_r$ and $\Sigma(\Omega_r) = \Sigma(\Omega)$.
By Theorem~5.3 in \cite{ACNS}, $\Omega$ and $\Omega_r$ have the same homotopy class as
$\Sigma(\Omega)$. Since $\Omega$ is connected by assumption, then also 
$\Sigma(\Omega)$ and $\Omega_r$ are connected.
\end{remark}

\begin{remark}{\it (About the regularity assumption (H2)).} 
When $\Omega$ is convex, assumption (H2) is satisfied because $\overline{\Omega_r}$ agrees with $\Omc{r}$ for every $r\in [0, \inradius)$. 
For general $\Omega$,  we have the inclusion $\overline{\Omega_r} \subseteq \Omc{r}$, which may be possibly strict (see for instance Example  \ref{e:nonreg} below). 
Assumption (H2) can be also rephrased by asking that  the set $C:= \Omc{\frac 1 \lambda}$ satisfies 
$C = \overline{\interior C}$. In topology, sets satisfying this last condition are known as \textsl{regular closed} sets.
It is clear from the definition that such sets are closed in the usual sense, and have a non-empty interior if they are not empty. \end{remark}

Assumptions (H1) and (H2) are sharp, as we can have multiple genuine solutions as soon as 
 $\Omega_{ \frac 1 \lambda}$ is not connected and/or 
$\overline{\Omega_{\frac 1 \lambda}} \neq \Omc{\frac 1 \lambda} $.
This fact is illustrated in Examples~\ref{e:nonconn} and~\ref{e:nonreg} below.

\begin{example}[Multiplicity of genuine solutions without (H1)]\label{e:nonconn}
\begin{figure}
	\includegraphics{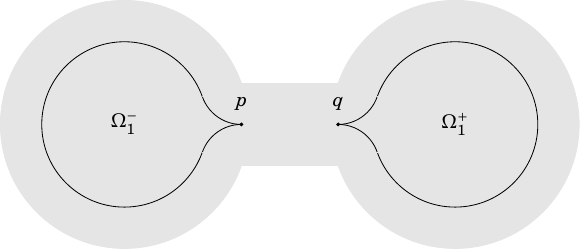}
		\caption{The set $\Omega$ of Example~\ref{e:nonconn} (grey) and $\partial\Omega_1$ (black)}\label{fig:nonconn}
\end{figure}

If $\Omega_{\frac 1 \lambda}$ is not connected, then  problem
\ref{f:P} may have more than one genuine solution.
Let us show this phenomenon with an explicit example.
Let $\Omega\subset\R^2$ be the set
\[
\Omega := B_3((-4,0)) \cup B_3((4,0)) \cup ((-4,4)\times(-1,1))
\]
(see Figure~\ref{fig:nonconn}), and let $\lambda = 1$.

The set $\Omega_1$ is not connected, since it is the disjoint union of
two connected components $\Omega_1^- := \Omega_1 \cap \{x_1 < 0\}$ and 
$\Omega^+_{1} := \Omega_1 \cap \{x_1 > 0\}$.

We have proved in Theorem~\ref{t:B} that the function $w_1$ is a solution to $(P)_1$. 

Furthermore, we claim that the infinity--harmonic potentials of $\overline{\Omega_1^{\pm}}$
relative to $\Omega$
are both solutions to $(P)_1$.
Let us prove this claim when $u$ is the infinity--harmonic potential of $\overline{\Omega_1^{-}}$.
By Proposition~\ref{p:pot} we have that
$u(x) = w_1(x)$ on the set
\[
A^- := \{x = (x_1, x_2)\in\Omega\setminus \overline{\Omega_1}:
\ x_1 < 2 \sqrt{2} - 4\}. 
\]
Hence, we already know that $u$ satisfies the free boundary condition of Definition~\ref{d:desilva}
at all points $x_0 \in F(u) = \partial\Omega_1^-$, $x_0 \neq p := (2\sqrt{2}-4, 0)$.
It remains to prove that the free boundary condition is satisfied at $p$.
Since $p$ has two projections 
$y_\pm := (2\sqrt{2}-4, \pm1)$
on $\partial\Omega$, it does not exist a smooth function $\varphi$
such that $u \prec_p \varphi^+$.
On the other hand, if $\varphi$ is a smooth function such that
$\varphi^+ \prec_p u$, then necessarily $|\nabla\varphi(p)| \leq 1$,
since $u(x)\leq \dist(x, \overline{\Omega_1^-})$.
This proves that $u$ is a solution to $(P)_1$.

One can also construct infinitely many other genuine solutions to $(P)_1$.
Specifically, let $q := -p$, let $C$ be a closed subset of 
$[p,q] \cup \overline{\Omega_1^+}$ with empty interior,
and let $K := C\cup\overline{\Omega_1^-}$.
Then the infinity--harmonic potential of $K$ relative to $\Omega$
turns out to be a solution to  $(P)_1$.  Another
symmetric family of genuine solutions can be constructed
by taking  $C$ a closed subset of 
$[p,q] \cup \overline{\Omega_1^-}$ with empty interior
and $K := C\cup\overline{\Omega_1^+}$.
(For both families, the free boundary condition can be checked by arguing with minor modifications
as done in the proof of Proposition \ref{p:trivial}).
\end{example}

\begin{example}[Multiplicity of genuine solutions without (H2)]\label{e:nonreg}

\begin{figure}
	\includegraphics{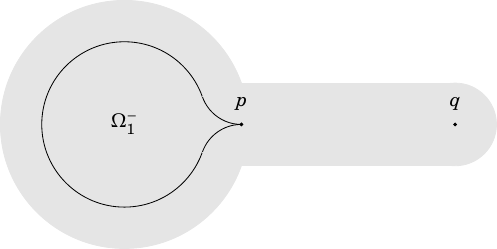}
	\caption{The set $\Omega$ of Example~\ref{e:nonreg} (grey) and $\partial\Omega_1$ (black)}\label{fig:nonreg}
\end{figure}

More than one genuine solution may occur
also in case $\overline{\Omega_\frac{1}{\lambda}}$ is strictly contained in  $\Omc{\frac{1}{\lambda}}$. To enlighten this fact, 
let us  modify the above example by 
considering the set
\[
\Omega := B_3((-4,0)) \cup B_1((4,0)) \cup ((-4,4)\times(-1,1)). 
\]
Again, we take $\lambda = 1$. In this case, 
$\{d \geq 1\} = \overline{\Omega_1} \cup [p,q] \neq \overline{\Omega_1}$,
with $p := (2\sqrt{2}-4, 0)$  and $q: = (4, 0)$.
Similarly as above, for every closed subset $C$ of the segment $[p,q]$,
the infinity--harmonic potential of $K := \overline{\Omega_1^-} \cup C$ relative to $\Omega$ is a
solution to \ref{f:P}.	
\end{example}

We now turn to the proof of Theorem \ref{t:unique}. It is based on the characterization of the set $\interior \{u=0\}$
(see Proposition~\ref{p:reg} below).
We start by proving a simple geometric lemma.

\begin{lemma}\label{l:K}
Let 
$A$ be a non-empty open subset of $\Omega$ such that, for some 
constant $R>0$,
\begin{equation}\label{f:A}
d(x) = \dist(x, \partial A) + R,
\qquad \forall x\in A.
\end{equation}
Then $A$ is a union of connected components of $\Omega_R$. 
In particular, if $\Omega_R$ is connected,
then $A = \Omega_R$.
\end{lemma}

\begin{proof}
From \eqref{f:A} we have that $d(x) > R$ for every $x\in A$,
hence $A\subseteq\Omega_R$.

We claim that $\partial A \subseteq \partial\Omega_R$.
Specifically, let $y\in\partial A$. 
For every $\varepsilon>0$ there exists a point $x\in A$ such that
$|x-y| < \varepsilon$, so that, by \eqref{f:A},
\begin{gather*}
d(y) <
d(x) + \varepsilon = \dist(x,\partial A) + R + \varepsilon < R + 2\varepsilon,
\\
d(y) >
d(x) - \varepsilon = \dist(x,\partial A) + R - \varepsilon 
> R - 2\varepsilon,
\end{gather*}
hence $d(y) = R$, and the claim is proved.

Let $A'$ be a connected component of $A$, and let $B$ a connected component of $\Omega_R$
such $A' \cap B \neq \emptyset$.
By the previous claim, $\partial A'\cap B = \emptyset$, hence $B$ can be written as the
union of the two open sets $A'$ and $B\setminus\overline{A'}$.
Since $B$ is connected, then necessarily
$B\setminus\overline{A'} = \emptyset$ and $A'=B$.
\end{proof}

\begin{proposition}\label{p:reg}
	Let $\lambda > 1/\inradius$ and let $u$ be a solution to \ref{f:P}.
	Then $\interior \{u = 0\}$ is a 
	(possibly empty) union of connected components of $\Omega_{\frac{1}{\lambda}}$.
\end{proposition}

\begin{proof}
	We are going
	to prove that, if the set $A := \interior \{u = 0\}$ is not empty, it satisfies the assumption \eqref{f:A} of Lemma~\ref{l:K}
	with $R = \frac{1}{\lambda}$.
	
	Let $x\in A$, let $x_0\in\Pi_{\partial A}(x)$
	and let $r := |x-x_0|$, so that $B_r(x)\subset A$ and $x_0\in F(u)$.
	Let us consider the function
	\[
	\varphi(y) := \frac{|y-x|-r}{d(x) - r},
	\qquad
	y\in \Omega.
	\]
	We have that $\varphi(y) \geq 0$ for every $y\in \partial A \subseteq F(u)$,
	and
	$\varphi(y) \geq 1$ for every $y\in\partial\Omega$.
	Hence, by comparison, $\varphi\geq u$ in $\Omega\setminus A$ and,
	in particular, $u \prec_{x_0} \varphi^+$.
	By Definition~\ref{d:desilva}, it follows that
	\begin{equation}\label{f:dis1}
	|\nabla\varphi(x_0)| = 
	\frac{1}{d(x) - r} \geq \lambda.
	\end{equation}
	Let $z\in\Pi_{\partial\Omega}(x)$.
	The point $y_0 := x + r \, \frac{z-x}{|x-z|}$ belongs to $\overline{B}_r(x)\subset\overline{A}$
	and, by Proposition~\ref{c:gradbound},
	$\overline{B}_r(x) \subseteq \overline{A} \subseteq \overline{\Omega}_{\frac{1}{\lambda}}$,
	so that 
	\begin{equation}\label{f:dis2}
	\frac{1}{\lambda} \leq d(y_0) = |y_0 - z| = |x-z| - |x-y_0|
	= d(x) - r \,.
	\end{equation}
	From \eqref{f:dis1} and \eqref{f:dis2} it follows that
	$d(x) - r = \frac{1}{\lambda}$,
	i.e.\ the assumptions of Lemma~\ref{l:K} hold with $R = \frac{1}{\lambda}$.
\end{proof}

\medskip
We are now in a position to give:

\medskip
{\it Proof of Theorem \ref{t:unique}}. 
	Let $u$ be a genuine solution to \ref{f:P}, for some $\lambda > 1/\inradius$. 
	
	Since by assumption  
	the interior of $\{u = 0\}$ is not empty, by Proposition \ref{p:reg} it is a union of connected components of  $\Omega_{\frac 1 \lambda}$ and hence, by assumption (H1), it agrees with $\Omega _{\frac 1 \lambda}$. 
		
	On the other hand, by Proposition~\ref{c:gradbound},  the closed set  $\{u=0\}$ is  contained in 
	$\Omc{\frac 1 \lambda}$ and, by assumption (H2), we have $\Omc{\frac 1 \lambda} = \overline{\Omega_{\frac 1 \lambda}}$.
	
Summarizing, we have 
$$ \Omega _{\frac 1 \lambda}  = \interior  \big ( \{ u = 0 \}  \big ) \subseteq \{ u = 0 \} \subseteq 
\Big \{ d \geq \frac{1}{\lambda} \Big \} 
= \overline{\Omega_{\frac 1 \lambda}}\,.$$

Hence, $\{u = 0\} = \overline{\Omega_{\frac{1}{\lambda}}}$ and $u = w_{\frac{1}{\lambda}}$.
\qed

\section{Characterization of solutions} 
\label{s:char}

In the following theorem we will characterize all solutions to \ref{f:P}
as the infinity-harmonic potentials of compact subsets of $\Omega$.

\begin{definition}\label{d:kl}
For a fixed $\lambda \geq 1/\inradius$,
let $\mathcal{K}_\lambda$ be the family of all non-empty sets $K\subset\R^n$ 
satisfying the following properties:
\begin{itemize}
	\item[(i)]
	$K$ is a compact subset of $\Omc{1/\lambda}$.
	\item[(ii)]
	If $\widetilde{K}$ is a connected component of $K$
	with non-empty interior,
	then $\interior \widetilde{K}$ coincides with a
	connected component of $\Omega_{1/\lambda}$.
	\item[(iii)] 
	If $\Omega\setminus K$ is decomposed as in Section~\ref{ss:potentials},
		then the set $I_0$ defined in \eqref{f:I0} is empty ({\it cf.} Lemma \ref{l:clarify} for equivalent conditions). 
	\end{itemize}	
\end{definition}

\begin{theorem}\label{t:solPl2}
	Let $\lambda \geq 1/\inradius$.
Then a function $u\in C(\overline{\Omega})$ is a non-constant 
solution to~\ref{f:P}
if and only if there exists a set $K\in \mathcal{K}_\lambda$ such that
$u = w_K$.
\end{theorem}

\begin{proof}
The case $\lambda = 1/\inradius$ is trivial
(see Remark~\ref{r:inradius}),
so that we shall assume that $\lambda > 1/\inradius$.

\medskip
Let $u\in C(\overline{\Omega})$ be a solution to \ref{f:P}.
Let us prove that the set $K := \{u=0\}$ belongs to the class $\mathcal{K}_\lambda$ introduced in Definition \ref{d:kl},
and that $u = w_K$.

Condition (i) is satisfied 
by Proposition~\ref{c:gradbound}. 

Condition (ii) is clearly satisfied if $u$ is a non-genuine solution, while it  follows from Proposition~\ref{p:reg}
 if $u$ is a genuine solution. 

Condition (iii) can be easily checked arguing by contradiction. 
Specifically, assume that the set $I_0$ defined in \eqref{f:I0} is not empty.
In this case, there exists a connected component $A$ of $\Omega\setminus K$ such
that $\partial A\subset K$. But then, by uniqueness, necessarily $u=0$ on $A$, with $A\cap K=\emptyset$,
against the definition of $K$. 

We have thus proved that $K \in \mathcal K _\lambda$. Finally we observe that, since $K$ satisfies condition~(iii), we have $\{w_K = 0 \} = K$ 
(cf.\ Lemma~\ref{l:clarify}), and hence $u = w_K$.

\medskip
Vice versa,  let $K\in \mathcal{K}_\lambda$
and let us prove that $w_K$ is a solution to \ref{f:P}.

By property (iii) in Definition \ref{d:kl}, we have that $F(w_K) = \partial K$, 
hence it is enough to prove that
the free boundary condition is satisfied at any point of $\partial K$.
Let $x_0 \in \partial K$. 

We have two possibilities:
either $x_0 \not\in \overline{\interior K}$,
or $x_0 \in \partial  B$, where $B$ is a connected component of $ \interior K$ (which thanks to property (ii) in Definition \ref{d:kl} is also a connected component of  the open set $\Omega_{1/\lambda}$).

If $x_0 \not\in \overline{\interior K}$,
we are done by arguing exactly as in Proposition~\ref{p:trivial}
(in particular, by exploiting property (i) in Definition \ref{d:kl}).

If $x_0 \in \partial B$ , 
we argue as in the proof of Theorem~\ref{t:B}(a).
More precisely, we prove firstly that the following inequalities analogous to \eqref{f:cfdist} are satisfied:
\begin{equation}\label{f:cfdist2}
1 - \lambda\, d(x)
\leq w_K(x)
\leq \lambda\, \dist(x, \partial B),
\qquad
\forall x\in \overline{\Omega}\setminus K, 
\end{equation}
with equalities for every
$x\in [x_0, y_0]$, being $y_0 \in \Pi_{\partial\Omega}(x_0)$. 
Then, by using \eqref{f:cfdist2}, 
we  obtain the free boundary condition at $x_0$ by proceeding in the same way as in the second part of the proof
of Theorem~\ref{t:B}(a).  
\end{proof}

\section{Asymptotics of $p$-Bernoulli problems as $p\to +\infty$}

In this section we explore the relation between problem \ref{f:P} and the interior Bernoulli problem for the $p$-Laplacian. 
For the benefit of the reader, we start by revisiting in more detail some facts which in part have been already mentioned in the Introduction
(some bibliographical references already given  therein are skipped below). 
 
The  interior Bernoulli free boundary problem for the $p$-Laplacian, for a given $p>1$,
consists in finding a (non-constant) solution to
\begin{equation}\label{f:plapconv}
\begin{cases}
\Delta_p u = 0
& \text{ in } \Omega^+(u)\,, 
\\
u = 1 & \text{ on } \partial \Omega,
\\ 
|\nabla u| = \lambda    
& \text{ on } F(u)\,.
\end{cases}
\end{equation}

Then  the {\it Bernoulli constant for the $p$-Laplacian} is defined by
\[
\lp := \inf\{\lambda > 0:\ \text{\eqref{f:plapconv} admits a non-constant solution}\}.
\]

Here a solution to \eqref{f:plapconv} is meant as a function $u\in W ^ {1,p} _1 (\Omega)$ such that, according to \cite{DanPet}, 
the free boundary condition 
is satisfied in the following weak sense:
\begin{equation}\label{f:weakFB}
\lim_ {\e \to 0 } \int _ {\partial 
\{ u > \e \} \cap \Omega} (  |\nabla u | - \lambda )  \eta \cdot \nu  = 0 
\qquad \forall \eta \in W ^ {1, p } _0(\Omega; \R ^n)\,,
\end{equation}
where $\nu$ is the unit outward normal.

In particular,  when $\Omega$ is a regular convex domain, the following results due to Henrot and Shahgholian hold: 

\smallskip
$\bullet$ 
for every $\lambda \geq \lp$, problem \eqref{f:plapconv}
	admits a classical non-constant solution $u\in C(\overline{\Omega^+(u)}) \cap C^2(\Omega^+(u))$, which has convex level sets; moreover,  the free boundary $F(u)$ is of class $C^{2,\alpha}$ \cite[Thm.~2.1]{HS}, and the free boundary condition is satisfied in the pointwise sense  
\[
\lim_{\Omega^+(u) \ni y \to x}
|\nabla u(y)| = \lambda \qquad \forall x \in F(u)\,.
\]

\smallskip

$\bullet$ $\lp$ can be characterized, loosely speaking, as the infimum of positive $\lambda$ such that the family of sub-solutions
to \eqref{f:plapconv} is not empty, and it satisfies the lower bound 
\begin{equation}\label{lblp}
\lp \geq 1/\inradius
\end{equation} \cite[Thms.~3.1 and 3.2]{HeSh2}. 

\medskip

\medskip
When $\Omega$ is an arbitrary domain, not necessarily convex, following the celebrated work   \cite{AlCa} by Alt and Caffarelli, 
in order to find solutions to problem \eqref{f:plapconv} 
one can consider the integral functionals
\[
J ^{\lambda } _p  (u) := \frac{1}{p} \int_\Omega \left(\frac{|\nabla u|}{\lambda} \right)^p + \frac{p-1}{p} \,  \big | \{ u >0 \}   \big | 
\]
and look for minimizers to 
\begin{equation}\label{f:minp}
\min \Big \{ J ^{\lambda} _p (u)  \ :\ u \in W ^ {1,p}_1 (\Omega) \Big \} \,,
\qquad
W ^ {1,p}_1 (\Omega) := 1 + W ^ {1,p}_0 (\Omega).
\end{equation} 
Accordingly, the constant
\[
\Lp := \inf\{\lambda > 0:\ \text{\eqref{f:minp} admits a non-constant solution}\}, 
\]
can be regarded as a \textit{variational Bernoulli constant for the $p$-Laplacian}.  
We have that:

\smallskip
$\bullet$ 
For every $\lambda \geq \Lp$, problem \eqref{f:minp}
	admits a non-constant minimizer $u$  (see \cite[Thm.~1.1]{DaKa2010}); such minimizer turns out to be a solution to the Bernoulli problem
	\eqref{f:plapconv}, provided the free boundary condition 
$|\nabla u| = \lambda$ is understood in the weak sense \eqref{f:weakFB} \cite[Thm.~2.1]{DanPet}; moreover, the free boundary $F(u)$ is a locally analytic hyper-surface, except for a $\mathcal H ^{n-1}$-negligible singular set \cite[Cor.~9.2]{DanPet}. 
	
	\smallskip
$\bullet$ As a consequence of the results recalled at the above item,
we have that 
\begin{equation}\label{f:Ll}
\Lp \geq \lp\,;
\end{equation}
 this inequality may be strict, as
 the explicit computation of both constants $\Lp$ and $\lp$ in case of the ball reveals \cite[Section 4]{DaKa2010}.

\medskip

Being this  a quick picture of the state of the art, 
in the light of the results proved in the previous sections for problem \ref{f:P}, it
is natural to ask:

\smallskip
{\it What is the asymptotics of the Bernoulli constants $\lp$ and $\Lp$ as $p \to + \infty$?  Further, 
if for a fixed $\lambda$ and $p$ large enough
there exists a non-constant solution $u_p$ to \eqref{f:plapconv}, 
what is the limiting behaviour of $u _p$ as $p \to + \infty$?}

\medskip
Regarding the asymptotics of the  Bernoulli constants $\lp$ and their variational counterparts $\Lp$,   we have: 

\begin{proposition}\label{p:convl}
Let $\Omega\subset\R^n$ be a bounded open domain. 
Then
\[
\limsup_{p \to + \infty} \lp
\leq
\limsup_{p \to + \infty} \Lp  \leq  1/ \inradius\,.
\] 
\end{proposition}

\begin{proof}
In view of the inequality \eqref{f:Ll}, it is enough to prove that
\[
\limsup_{p \to + \infty} \Lp   \leq 1/ \inradius\,.
\] 
To obtain this inequality we observe that, if we fix $\l > 1/\inradius$, for $p$ large enough problem~\eqref{f:minp} admits a non-constant minimizer. Indeed, setting $v _\l := ( 1- \l d) _+$, for $p\gg 1$ we have
\[
J ^\lambda _p (v_\l ) = \frac{1}{p}  |D_{1/\l} |  + \frac{p-1}{p}   |D_{1/\l}| 
=  |D_{1/\l} | < \frac{p-1}{p}  |\Omega | = J _ p (1)\,.
\qedhere
\] 
\end{proof}

\begin{corollary}\label{c:convl}
Assume that $\Omega$ is convex with $\partial \Omega$ of class $C ^1$. Then
\[
\lim _{p \to + \infty} \lp = \lim _{p \to + \infty} \Lp   = 1/ \inradius\,.
\] 
\end{corollary}

\begin{proof}
It follows from \eqref{lblp} and Proposition \ref{p:convl}.
\end{proof}

\bigskip
Now, let $\Omega$ be convex and let  $\l > 1/\inradius$.   
By  Proposition \ref{p:convl}, for $p$ large enough there exists a non-constant solution $u _p$ to 
\eqref{f:plapconv}. Moreover, by Corollary \ref{c:convex}, problem \ref{f:P}  admits a unique solution given precisely by the infinity harmonic potential $w _{\frac{1}{\lambda}}$ of $\overline{\Omega _{\frac{1}{\lambda}}}$. 
Nevertheless,  we cannot expect that, in general, $u _p$ converge to $w _{\frac{1}{\lambda}}$ as $p \to + \infty$. To enlighten this fact and get a feeling of the situation, let us have a look at what happens
when $\Omega$ is a ball.

\begin{example}[The radial case]
Let $B_R$ be the ball of center $0$ and radius $R$ in $\R ^n$, and let $\l > 1/R$. 
It is well-known that for $\lambda = \lambda _p ( B _R)$ the Bernoulli problem \eqref{f:plapconv} on $B _R$ admits a unique solution, which is called parabolic, whereas for any $\lambda > \lambda _p ( B _R)$ it admits two solutions, which are called  hyperbolic and elliptic (as they are respectively decreasing and increasing with respect to the parameter $\lambda$).  

Since we want to examine the asymptotic behaviour of these solutions as $p \to + \infty$, let us briefly recover their expressions.  
By a result of Reichel \cite{reichel}, a solution to problem  \eqref{f:plapconv} on $B _R$ is necessarily radial. 
 Hence,  for $\rho\in (0, R)$ and $p > n$, we compute the $p$-harmonic function $u_p$ in $B_R\setminus\overline{B}_\rho$
which satisfies the Dirichlet boundary conditions $u_p = 1$ on $\partial B_R$ and $u_p = 0$ on $\partial B_\rho$. It is given by
\begin{equation}\label{f:up}
u_p(x) = \frac{|x|^\alpha - \rho^\alpha}{R^\alpha - \rho^\alpha}\,,
\quad \rho < |x| < R,
\qquad
\alpha := \frac{p-n}{p-1}
\end{equation}
(observe that, for $p>n$, the exponent $\alpha$ belongs to $(0, 1)$, and tends to  $1$ as $p\to +\infty$).

We are interested in finding the values of $\rho\in (0,R)$ such that
\begin{equation}\label{f:pbd}
|\nabla u_p(x)| =  \lambda,
\qquad \text{for}\ |x| = \rho\,.
\end{equation}
Since 
$
|\nabla u_p(x)| = \alpha\, \frac{|x|^{\alpha-1}}{R^\alpha - \rho^\alpha}
$, condition \eqref{f:pbd} is equivalent to 
\begin{equation}\label{f:zerof}
f_\alpha(\rho) := \lambda\, \rho^\alpha + \alpha\, \rho^{\alpha -1} - \lambda\, R^\alpha = 0.
\end{equation}
It is immediate to check that $f_\alpha$ is strictly decreasing in 
$\left(0, \frac{1-\alpha}{\rho}\right)$
and strictly increasing in $\left(\frac{1-\alpha}{\rho}, R\right)$, so that 
\[
m_\alpha := \min_{(0, R)} f_\alpha
= f_\alpha\left(\frac{1-\alpha}{\rho}\right)
= \left(\frac{\lambda}{1-\alpha}\right)^{1-\alpha} - \lambda\, R^\alpha.
\]
Moreover,
\[
\lim_{\rho\to 0+} f_\alpha(\rho) = +\infty,
\qquad
f_\alpha(R) = \alpha\, R^{\alpha -1} > 0.
\]
Hence, equation \eqref{f:zerof} has one solution
if $m_\alpha = 0$, two solutions if $m_\alpha < 0$,
no solutions if $m_\alpha > 0$.
Observe that
\[
m_\alpha \leq 0
\quad\Longleftrightarrow\quad
\lambda \geq \lambda_p (B_R):= \frac{1}{R}\left(1-\alpha\right)^{1-1/\alpha}
= \frac{1}{R} \left(\frac{n-1}{p-1}\right)^{-(n-1)/(p-n)}\,.
\]
In particular, for $p$ large enough, 
since $\lim_{\alpha\to 1^-} m_\alpha = 1 - \lambda\, R < 0$,
equation \eqref{f:zerof}  
has exactly two zeros
$\rho '_\alpha$ and $\rho''_\alpha$; correspondingly, 
the sets $\partial B _{\rho'_\alpha}$ and $\partial B _{\rho''  _\alpha}$ are  the free boundaries of the so-called hyperbolic and 
elliptic solutions to~\eqref{f:plapconv}.

Let us look at what happens as $p\to + \infty$. We know from the above computations that 
\[
0 < \rho' _\alpha < \frac{1-\alpha}{\lambda} < \rho'' _\alpha < R.
\]
This gives at once 
$\rho'_\alpha \to 0$ as $\alpha\to 1^-$. On the other hand it is easily seen that, for every
$0 < \varepsilon < \min\{ 1/\lambda, R - 1/\lambda\}$, it holds
\[
\lim_{\alpha\to 1-} f_\alpha\left(R - \frac{1}{\lambda}-\varepsilon\right) = -\varepsilon\, \lambda < 0,
\qquad
\lim_{\alpha\to 1-} f_\alpha\left(R - \frac{1}{\lambda}+\varepsilon\right) = \varepsilon\, \lambda > 0,
\]
so that
$\rho'' _\alpha \to R - \frac{1}{\lambda}$ 
as $\alpha\to 1^-$. 

Summarizing, the above analysis shows that the two families of $p$-harmonic functions which 
fit the Bernoulli free boundary condition 
\eqref{f:pbd} for
$\rho=\rho'_p$ and $\rho=\rho''_p$ have respectively the following asymptotic behaviour:   their
free boundary degenerates or converge to the set $\partial \Omega _{\frac{1}{\l}}$, i.e., 
\[
\rho'_p \to 0 \, , \qquad \quad 
\rho''_p \to R - \frac{1}{\lambda}\,,
\]
and the functions $u _p$, as given by \eqref{f:up},  converge  to
\[
w_R(x) = \frac{|x|}{R}\,,
\quad x\in \overline{B}_R,
\qquad \quad
w_{\frac 1\lambda}(x) = 1 - \lambda(R-|x|),
\quad x\in \overline{B}_R\setminus B_{R- \frac{1}{\lambda}}.
\]
In particular, only the elliptic family converges to the unique solution of \ref{f:P}. Let us remark that, for $\lambda \geq \L _p ( B _R)$,  
contrary to the hyperbolic solutions, 
the elliptic ones are variational, namely 
they  solve the minimization problem \eqref{f:minp} on $B _R$ (see \cite[Sec.~5.3]{FR97}, \cite[Sec.~4]{DaKa2010}).

\end{example} 

\bigskip

Now, as suggested by the example of the ball, we give a convergence result for variational solutions. 
The reader may find a similar $\Gamma$-convergence result in the paper \cite{KawSha}, where the authors deal with the asymptotic behaviour of variational energies related to 
the exterior Bernoulli boundary problem  for the $p$-Laplacian as $p \to + \infty$.

\begin{lemma}
\label{l:monot}

(i) 
For every function $u$ belonging to the space 
\[
\Vspace(\Omega) :=
\left\{
u \in C(\overline{\Omega}) \cap W^{1,\infty}(\Omega):\
u = 1\ \text{on}\ \partial\Omega
\right\}\,.
\]
the map $p \mapsto  J^\lambda_p (u)$ is monotone nondecreasing.

(ii) In the limit as $p \to + \infty$, the sequence
$(J^\lambda_p)_p$ $\Gamma$-converges, with respect to the weak topology of $W^{1,q}(\Omega)$, 
to  the functional  functional $J _\infty$ given by
\begin{equation}\label{f:jinf}
J _\infty (u) :=
\begin{cases}
   |\{u>0\}|,& \text{if}\ \|\nabla u\|_\infty \leq \lambda,\\
+\infty, &\text{otherwise}.
\end{cases}
\end{equation}
\end{lemma}

\begin{proof} 
The first part of the statement can be found in  \cite[Proposition 1]{KawSha}, but we enclose a short proof for the sake of completeness. 
 If $1<p\leq  q $, by applying Young's inequality 
$AB \leq (A^r/r ) + ( B ^ s/s)$, with $A = |\nabla u| /\l$, $B = 1$, $r = q/p$ and $s = r/ (r-1)$, we obtain 
\[
J^\lambda_p (u) \leq
\frac{1}{q}\int_\Omega \left(\frac{|\nabla u|}{\lambda}\right)^q\,dx
+ \left(\frac{q-p}{pq} + \frac{p-1}{p}\right)
\, |\{u > 0\}| =
J^\lambda_q (u)\,.
\]

The second part of the statement follows from the first one: it is enough to observe that
the functional $J _\infty$ is the ``pointwise'' limit of $J ^ \lambda _p$,  and 
apply a well-known property of $\Gamma$-convergence (see
\cite[Remark~1.40(ii)]{Brai}).
\end{proof}

\begin{theorem}\label{t:convp}
Let $\l > 1/\inradius$. 
For $p$ large enough, let $u _{\l, p }  $ be a solution to the $p$-Bernoulli problem ~\eqref{f:plapconv} 
which is found by solving the minimization problem~\eqref{f:minp}. 
Then, there exists an increasing sequence $(p_j)$, diverging
to $+\infty$,
such that
\[
u_{\l,p_j}   \rightharpoonup 
u_\infty\
\text{weakly in}\ W^{1,q}(\Omega) \quad \forall q>1\, , 
\qquad
u_{\l, p_j}  \to
u_\infty\
\text{uniformly in}\ \overline{\Omega}\,,
\]
where 
$u_\infty$ is a solution
of the $\infty$-Bernoulli problem \ref{f:P}
satisfying 
\[
\interior\{u_\infty = 0\} = \{d > 1/\lambda\}.
\]
\end{theorem}

\begin{proof} 
Thanks to the assumption $\l> 1/\inradius$ and to Proposition \ref{p:convl}, we know that for $p$ large enough problem \eqref{f:minp} admits a solution $u _{\lambda, p}$. 
As $\l$ is fixed, we shall write for brevity $J _p$ in place of $J ^ \l _p$ and $u _p$ in place of $u _{\l, p}$. 

Let us first show that, for every fixed $q>1$, the family $(u _p)$ is uniformly bounded in $W^{1,q}(\Omega)$. 
Since
\[
\frac{1}{p}\int_\Omega\left(\frac{|\nabla u_p|}{\lambda} \right)^p\, dx 
\leq J_p (u_p)
\leq J_p (1) =  \frac{p-1}{p}  |\Omega|
\leq  |\Omega|,
\]
we get 
\[
\|\nabla u _p \|_p \leq
\lambda p^{1/p}  |\Omega|^{1/p}.
\]
Then, by H\"older's inequality, 
for every $p > q+1$ it holds
\begin{equation}\label{f:estigupa}
\|\nabla u _p\|_q \leq \|\nabla u _p \|_p \, |\Omega|^{\frac{1}{q}-\frac{1}{p}}
\leq 
\lambda p^{1/p}  |\Omega|^{\frac{1}{q}}
\leq C,
\end{equation}
where $C>0$ is a constant independent of $p$.

Using a diagonal argument, we can construct an increasing sequence $p_j\to +\infty$
such that $u _{p _j}$ converges weakly in $W ^ {1,q} (\Omega)$ for every $q>1$ and uniformly in $\overline \Omega$
to a function $u_\infty$. 
We claim that $u_\infty$ is a solution to \ref{f:P}, satisfying 
$
\interior\{u_\infty = 0\} = \{d > 1/\lambda\}$.

\smallskip
The fact that $u_\infty$ is infinity harmonic in its positivity set is a standard consequence of the fact that 
$u _ {p_j}$ solve \eqref{f:minp} with $p = p_j \to + \infty$, see for instance the arguments in \cite[proof of Theorem 1]{RossiTeix}. 

Moreover, since $u _{p_j}= 1 $ on $\partial \Omega$ for every $j$, from the uniform convergence it follows immediately that
the same condition is satisfied by $u_\infty$.

Next we are going to show that  the set 
$$K := \{ u_\infty = 0\}$$
satisfies ${\rm int} (K) = \{ d> 1/\l \}$, and that it belongs to the class $\mathcal{K}_\lambda$ introduced in Definition~\ref{d:kl}.

From the  second inequality in \eqref{f:estigupa}, we see that $\|\nabla u _\infty \|_{L^\infty(\Omega)} \leq \lambda$.
 Since $u = 1$ on $\partial\Omega$, we conclude
that $u \in \Vspace(\Omega)$, where
$\Vspace(\Omega)$ is the space defined in Lemma~\ref{l:monot}.
Therefore, 
we deduce as  a first information on $K$ the inclusion
\begin{equation}\label{f:info1}
K \subseteq \{ d \geq 1/\lambda\}\,. 
\end{equation} 
To go farther, 
we claim that $u_\infty$ solves the minimization problem 
\begin{equation}\label{f:mina} 
\min \Big \{  J _\infty (u)  \ :\ 
u\in\Vspace(\Omega)
 \Big \} \,,
\end{equation} 
where $J_\infty$ is the functional defined by \eqref{f:jinf}.
Indeed, 
since $u _{p_j}\to u _\infty$ uniformly in $\overline{\Omega}$, for every fixed $\e>0$, 
there exists an index $j_\varepsilon\in\N$ such that
\[
|\{u_\infty  > 0\}| < |\{u _{p_j} > 0\}| + \varepsilon,
\qquad\forall j > j_\varepsilon\, . 
\]
Then, for $j> j_\varepsilon$, 
it holds
\begin{equation}\label{f:info2} 
\begin{array}{ll}
J _{p_j} (u _\infty)
& \displaystyle \leq \frac{1}{p _j}  | \{ u _\infty  >0 \} |   +  
 \frac{p_j-1}{p_j}  |\{u_\infty  > 0\}|
\\ \noalign{\medskip} & \displaystyle \leq
\frac{1}{p_j}   | \Omega |  + \frac{p_j-1}{p_j} (|\{ u _{p _j}  > 0\}| + \varepsilon)  
\\ \noalign{\medskip}  &  \displaystyle \leq
\frac{1}{p_j}  |\Omega|
+ J _{p _j} (u _{p _j})  + \varepsilon\,.
 \end{array} 
 \end{equation}

Thanks to  the monotonicity property stated in Lemma~\ref{l:monot},
we can now pass to the limit as $j \to + \infty$ in \eqref{f:info2}. 
 By the arbitrariness of $\e >0$, and recalling  that $u _{p _j}$ are solutions to~\eqref{f:minp} (with $p = p _j$), we obtain, for every $u \in \Vspace (\Omega)$, 
 \[
J _\infty (u_\infty ) = \lim_{j\to +\infty}  J _{p _j} (u_\infty ) \leq
\liminf_{j\to+\infty}   J _{p_j} (u _{p _j} )
\leq
  \liminf_{j\to+\infty}  J _{p _j} (u)
= J _\infty (u),
\]
so that $u _\infty$ solves problem \eqref{f:mina} as claimed. 
Consequently, by taking as a competitor the function $(1-\l d )_+$, we deduce that $|\{ u _\infty >0  \}| \leq |D_{\frac{1}{\lambda}}|$, or equivalently 
\begin{equation}\label{f:info3}
|K| \geq |\{ d \geq 1/\lambda\}|\,.
\end{equation}

Since $|\nabla d|=1$ a.e.\ in $\Omega$, every level set
of $d$ has zero Lebesgue measure, so that
$|\{d \geq 1/\lambda\}| = |\{d > 1/\lambda\}|$.
Since $\interior \{d \geq 1/\lambda\} = \{d > 1/\lambda\}$, 
by combining 
conditions~\eqref{f:info1} and~\eqref{f:info3}  we obtain that
$\interior K = \{d > 1/\lambda\}$.
As a consequence, $K$ belongs to the family
$\mathcal{K}_\lambda$ 
introduced in Definition~\ref{d:kl},
so that, by Theorem~\ref{t:solPl2},
$u_\infty$ is a solution to \ref{f:P}. 
\end{proof}

\begin{corollary}\label{c:convp}
Let $\lambda > 1/\inradius$.
Then, under the assumptions (H1)--(H2) of Theorem~\ref{t:unique}
(hence, in particular, when $\Omega$ is convex), 
in the limit as $p \to + \infty$ we have
\[
u_{\l,p}   \rightharpoonup 
w_{\frac{1}{\lambda}}\
\text{weakly in}\ W^{1,q}(\Omega) \quad \forall q>1\, , 
\qquad
u_{\l, p_j}  \to
w_{\frac{1}{\lambda}},\
\text{uniformly in}\ \overline{\Omega}\,,
\]
where 
$w_{\frac{1}{\lambda}}$ is the infinity harmonic potential  
of $\overline {\Omega _{\frac{1}{\lambda}}}$, namely the 
unique genuine solution to the $\infty$-Bernoulli problem \ref{f:P}. 
\end{corollary}

\begin{proof}
From Theorem~\ref{t:convp}, there exists an increasing sequence
$p_j \nearrow \infty$ such that
$u_{\lambda, p_j} \to u_\infty$,
with $u_\infty$ solution to \ref{f:P}.
Hence, by Theorem~\ref{t:unique} we have
that $u_\infty = w_{1/\lambda}$.
By the same argument, any other converging subsequence must converge
to $w_{1/\lambda}$.
\end{proof}

\bigskip 
{\bf Acknowledgments.} 
We thank Bozhidar Velichkov for a useful discussion about the viscosity interpretation of free boundary conditions. 

The authors have been supported by the Gruppo Nazionale per l'Analisi Matematica, 
la Probabilit\`a e le loro Applicazioni (GNAMPA) of the Istituto Nazionale di Alta Matematica (INdAM).

\def\cprime{$'$}

\end{document}